\documentclass[10pt,reqno]{article}

\usepackage[b5paper,top=1.2in,left=0.9in]{geometry}
\usepackage{verbatim}
\usepackage{amsthm}
\usepackage{amssymb}
\usepackage{amsmath}
\usepackage{graphicx}
\usepackage{appendix}
\usepackage{color} 
 \usepackage{float}
\usepackage{tikz}
\usetikzlibrary{arrows, calc, positioning}

\renewcommand{\d}{\mathrm{d}}
\newcommand{\D}{\mathrm{D}}

\newcommand{\e}{\mathrm{e}}

\newtheorem{Thm}{Theorem}[section]
\newtheorem{Lem}[Thm]{Lemma}
\newtheorem{Prop}[Thm]{Proposition}
\newtheorem{Cor}[Thm]{Corollary}
\newtheorem{Rem}[Thm]{Remark}
\newtheorem{Def}[Thm]{Definition}
\newtheorem{Ex}[Thm]{Example}
\newtheorem{Fact}[Thm]{Fact}
\newtheorem{Nota}[Thm]{Notation}

\newtheorem*{Con}{Conjecture}

\def\R{\mathbb{R}}
\def\Q{\mathbb{Q}}

\def\C{\mathbb{C}}

\def\T{\mathbb{T}}

\def\to{\longrightarrow}

\def\cD{\mathcal{D}}

\def\cH{\mathcal{H}}

\def\cP{\mathcal{P}}

\def\cU{\mathcal{U}}

\def\b{\beta}
\def\e{\epsilon}

\def\c{\gamma}
\def\D{\Delta}

\def\d{\delta}
\def\n{\eta}

\def\l{\lambda}

\def\s{\sigma}

\def\sl{\mathfrak{sl}}
\def\fb{\mathfrak{b}}
\def\g{\mathfrak{g}}
\def\fh{\mathfrak{h}}
\def\gl{\mathfrak{gl}}

\def\ox{\otimes}
\def\o+{\oplus}
\def\bo+{\bigoplus}
\def\x{\times}
\def\p[#1,#2]{\phi_{#1,#2}}
\def\til[#1]{\widetilde{#1}}
\def\what[#1]{\widehat{#1}}

\def\bS{\textbf{S}}

\def\z[#1]{z_{#1}}

\def\oo{\infty}
\def\=>{\Longrightarrow}

\def\<{\langle}
\def\>{\rangle}
\def\corr{\longleftrightarrow}
\def\^{\wedge}
\def\+{\dagger}

\def\inv{^{-1}}

\def\over[#1]{\overline{#1}}
\def\vec[#1]{\overrightarrow{#1}}
\def\mat[#1, #2]{\left[\begin{array}{ccccc}#1\end{array}\left|\begin{array}{c}#2\end{array}\right.\right]}
\def\xto[#1]{\xrightarrow{#1}}
\def\dd[#1,#2]{\frac{d#1}{d#2}}
\def\del[#1,#2]{\frac{\partial #1}{\partial #2}}
\def\Facts[#1]{\begin{Fact}\mbox{}\begin{itemize}#1\end{itemize}\end{Fact}}
\def\Notation[#1]{\begin{Nota}\mbox{}\begin{itemize}#1\end{itemize}\end{Nota}}
\def\Eqn[#1]{\begin{eqnarray*}#1\end{eqnarray*}}
\def\tab{\;\;\;\;\;\;}

\newcommand{\veca}[2][cccccccccccccccccccccccccccccccccccccccccc]{\left(\begin{array}{#1}#2 \\ \end{array} \right)}

\newcommand{\Eq}[1]{\begin{align}#1\end{align}}
\newcommand{\case}[2][cccccccccccccccccccccccccccccccccccccccccc]{\left\{\begin{array}{#1}#2 \\ \end{array}\right.}

\begin{document}

\title{Gauss-Lusztig Decomposition for $GL_q^+(N,\R)$ and Representation by $q$-Tori}
\author{  Ivan C.H. Ip\footnote{
          Kavli IPMU (WPI), 
          The University of Tokyo, 
		Kashiwa, Chiba 
		277-8583, Japan
\newline
          Email: ivan.ip@ipmu.jp}
          }
\date{}
\numberwithin{equation}{section}

\maketitle

\begin{abstract}
We found an explicit construction of a representation of the positive quantum group $GL_q^+(N,\R)$ and its modular double $GL_{q\til[q]}^+(N,\R)$ by positive essentially self-adjoint operators. Generalizing Lusztig's parametrization, we found a Gauss type decomposition for the totally positive quantum group $GL_q^+(N,\R)$ for $|q|=1$, parametrized by the standard decomposition of the longest element $w_0\in W=S_{N-1}$. Under this parametrization, we found explicitly the relations between the standard quantum variables, the relations between the quantum cluster variables, and realizing them using non-compact generators of the $q$-tori $uv=q^2 vu$ by positive essentially self-adjoint operators. The modular double arises naturally from the transcendental relations, and an $L^2(GL_{q\til[q]}^+(N,\R))$ space in the von Neumann setting can also be defined.
\end{abstract}


\newpage
\tableofcontents
\section{Introduction}

The goal of the present work is to give an explicit construction of a representation of the positive quantum group $GL_q^+(N,\R)$ and its modular double $GL_{q\til[q]}^+(N,\R)$ by positive essentially self-adjoint operators acting on certain Hilbert space $\cH$. This is done by finding a quantum analogue of the Gauss-Lusztig decomposition for $GL_q(N)$.

Let $G$ be a semi-simple group of simply-laced type, $T$ its $\R$-split maximal torus of rank $r$, and $U^{\pm}$ its maximal unipotent subgroup with $\dim U^+=m$. The Gauss decomposition of the max cell of $G$ is given by 
\Eq{ G=U^- T U^+.}
In type $A_r$ $(N=r+1)$, this amounts to the decomposition into lower triangular, diagonal and upper triangular matrices.

On the other hand, given a \emph{totally positive} matrix $G_{>0}$, where all entries of the matrix and the determinants of its minors are strictly positive, it can be decomposed as 
\Eq{ G_{>0}=U_{>0}^- T_{>0} U_{>0}^+,}
where all the entries and determinant of the minors of $U^{\pm}$ and $T$ are strictly positive if they are not identically zero. Lusztig in \cite{Lu} discovered a remarkable parametrization of $G_{>0}$ using a decomposition of the maximal Weyl group element $w_0\in W$. Let $w_0=s_{i_1}...s_{i_m}$ be a reduced expression for $w_0$, then there is an isomorphism between $\R_{>0}^m\to U_{>0}^+$ given by
\Eq{(a_1,a_2,...,a_m)\mapsto x_{i_1}(a_1)x_{i_2}(a_2)...x_{i_m}(a_m),}
where $x_{i_k}(a_k) = I_N+a_k E_{i_k,i_k+1}$ and $E_{i,j}$ is the matrix with 1 at the entry $(i,j)$ and 0 otherwise. Similar result also holds for $U_{>0}^-$. With this isomorphism Lusztig went on to generalize the notion of total positivity to Lie groups of arbitrary type.

In \cite{BFZ}, Berenstein et al. studied this decomposition for type $A_r$, in the context now known as \emph{cluster algebra}. They showed various relations and parametrizations using the cluster variables, in this case corresponding to determinant of the different minors. Corresponding to the canonical decomposition of $w_0$ is the parametrization using \emph{initial minors}, which are the determinants of those square sub-matrices that start from either the top row or the leftmost column. Using this parametrization, we found in \cite{FI} a family of positive principal series representations of the modular double $\cU_{q\til[q]}(\sl(N,\R))$, where the notion of the modular double was first introduced by Faddeev \cite{Fa1,Fa2} for $N=2$. These positive representations generalize the self-dual representations of $\cU_{q\til[q]}(\sl(2,\R))$ studied for example in \cite{BT,Ip,PT1}.

On the other hand, in order to study the quantum group $GL_q^+(2,\R)$ in the $C^*$-algebraic and von Neumann setting, in \cite{Ip, Pu} a quantum version of the Gauss decomposition for $GL_q(2)$ is studied, where any matrices are decomposed into product of the form
\Eq{\veca{z_{11}&z_{12}\\z_{21}&z_{22}}=\veca{u_1&0\\v_1&1}\veca{1&u_2\\0&v_2},}
where $u_iv_i=q^2v_iu_i$ are mutually commuting Weyl pair that generates the algebra of $q$-tori. In the split case, we set $|q|=1$, where 
\Eq{q=e^{\pi ib^2},\tab\til[q]=e^{\pi ib^{-2}}} with $b^2\in\R\setminus\Q$, $0<b<1$. Then the Weyl pair are represented by the canonical positive essentially self-adjoint operators 
\Eq{u=e^{2\pi bx},\tab v=e^{2\pi bp},} and the above decomposition gives a realization of the positive quantum group $GL_q^+(2,\R)$ where all entries and the quantum determinant are represented by positive essentially self-adjoint operators acting on $L^2(\R^2)$. Moreover, by replacing $b\to b\inv$ we obtain the representations for the modular double $GL_{q\til[q]}^+(2,\R)$. It is further shown in \cite{Ip} that $GL_q^+(2,\R)$ is the Drinfeld-Woronowicz's \emph{quantum double group} \cite{PW} over the quantum $ax+b$ group, and its harmonic analysis is studied in detail. A new Haar functional is discovered, and an $L^2$-space of ``functions" over $GL_{q\til[q]}^+(2,\R)$ is defined using this Haar functional. Then we proved that the regular representation of the modular double $U_{q\til[q]}(\sl(2,\R))$ on $L^2(GL_{q\til[q]}^+(2,\R))$ decomposes into direct integral of the positive principal series representations.

Combining the approaches above, our aim in this paper is to find the Gauss-Lusztig decomposition of the positive quantum group of higher rank, $GL_q^+(N,\R)$, in terms of the unipotent parameters $a_i$ defined above. These parameters are no longer commuting positive real numbers, and it is the goal of this paper to discover their quantum relations with each other, such that the decomposition gives precisely the definition of the quantum group $GL_q(N)$, and furthermore they are represented by positive essentially self-adjoint operators. Let us call two variables \emph{quasi-commuting} if they commute up to a power of $q^2$. In this paper we prove the following Theorem:

\begin{Thm}[Gauss-Lusztig Decomposition] The generators of the quantum group $GL_q^+(N,\R)$ can be represented by $N^2$ operators
$$\{b_{m,n}, U_k, a_{m,n}\}$$
with $1\leq n\leq m\leq N-1,1\leq k\leq N$, where each variable is positive self-adjoint operator that commutes or $q^2$-commutes with each other, so that
\begin{itemize} 
\item[(1)] The variables $\{U_k, a_{m,n}\}$ generate the upper triangular quantum Borel subgroup $T_{>0}U_{>0}^+$,
\item[(2)] The variables $\{b_{m,n},U_k\}$ generate the lower triangular quantum Borel subgroup $U_{>0}^-T_{>0}$,
\item[(3)] The variables $a_{m,n}$ commute with $b_{m,n}$,
\end{itemize}
Furthermore, the Gauss-Lusztig decomposition for the other parts of the modular double $GL_{\til[q]}(N,\R)$ can be obtained by replacing all variables $\{b_{m,n}, U_k, a_{m,n}\}$ by their tilde version 
\Eq{x\mapsto \til[x]:=x^{\frac{1}{b^2}}.}
\end{Thm}
As a corollary of the calculations, we also have the following results:
\begin{Thm}
\begin{itemize}
\item[(1)] There is an embedding of $GL_q^+(N,\R)$ into the algebra of $\lfloor\frac{N^2}{2}\rfloor$ $q$-tori generated by $\{u_i,v_i\}$ satisfying $u_iv_i=~q^2v_iu_i$, which are realized by
\Eq{u_i=e^{2\pi bx_i},\tab v_i=e^{2\pi bp_i}.}
\item[(2)] The quantum cluster variables $x_{ij}$, defined by the quantum determinant of the initial minors, can be represented as products of the variables $\{b_{m,n}, U_k, a_{m,n}\}$, and hence they quasi-commute with each other.
\end{itemize}
\end{Thm}

From the main Theorem, we can extend the quantum group $GL_{q\til[q]}^+(N,\R)$ into the $C^*$-algebraic setting by giving an operator norm to each element which is represented by integrals of continuous complex powers of the generators, completely analogous to the $N=2$ case. We can also give an $L^2$ completion and define the Hilbert space $L^2(GL_{q\til[q]}^+(N,\R))$. Then it is natural to conjecture its decomposition under the regular representation of the modular double $U_{q\til[q]}(\sl(N,\R))$ into the direct integral of positive principal series representations constructed in \cite{FI}, in analogy to the decompositions of $L^2(GL_{q\til[q]}^+(2,\R))$.

The Gauss decomposition for a general quantum group is definitely not new \cite{DKS, WZ}. However the usual notion in the context of $GL_q(N)$ is just decomposing the quantum group into a product of lower and upper triangular matrices, and the quantum Pl\"{u}cker relations between the coordinates are studied. Though this approach is a natural consideration, the relations involved are quite ad hoc, and furthermore it has no way to be generalized to the positive setting and its representation is rather unclear. Therefore we name our decomposition the \emph{Gauss-Lusztig decomposition} to distinguish it from the standard approach, where we decompose our quantum group into products of elementary matrices bearing a quantum variable, so that the positivity and their representations are manifest.

Finally we also remark that in \cite{BZ,FG2}, the notion of quantum cluster algebra is studied, where quasi-commuting cluster variables are considered, and the $q$-commuting relations are compatible with the algebraic framework. However its relations to the parametrization of $GL_q(N)$ is not very explicit, and its representation by the canonical $q$-tori $\{e^{2\pi bx},e^{2\pi bp}\}$ is not shown. In this paper, starting from the very definition of a quantum group, we found using new combinatorics method that these cluster variables, quasi-commuting in some complicated powers of $q^2$, are actually decomposed into simpler variables $\{b_{m,n}, U_k, a_{m,n}\}$ that commute only up to a factor of $q^2$, and explicit formula is given for the case $GL_q(N)$. The $q$-commutations we found explicitly are closely related to the Poisson structure of the cluster $\mathcal{X}$-variety considered in \cite{FG1}. We note that in this paper we only use a single choice of cluster variables given by the initial minors. A more thorough understanding of the theory of quantum cluster algebra in the context of quantum groups should be possible by also considering explicitly the quantum exchange relations to other clusters, corresponding to different parametrization of the maximal element $w_0$ explained in Theorem \ref{GLcomtrans}.

The paper is organized as follows. In Section \ref{sec:glq2} we describe in detail the Gauss decomposition for $GL_q(2)$ studied in \cite{Ip}. In Section \ref{sec:gln} we describe the Lusztig parametrization of the totally positive matrix in $GL^+(N,\R)$, and the description of the cluster variables defined in \cite{BFZ}. Then we introduce the definition of $GL_q(N)$ in Section \ref{sec:glqn}, and using certain combinatorics methods, we find in Section \ref{sec:gauss} the quantum relations between the variables of the Gauss-Lusztig decomposition. In Section \ref{sec:qtori} we constructed the representation of these quantum variables using $N^2-2$ quantum tori, and also present an example demonstrating the minimal representation using only $\lfloor \frac{N^2}{2}\rfloor$ tori. Finally using the quantum tori realization, in Section \ref{sec:posmod} we define the positive quantum group $GL_q^+(N,\R)$, and describe its relation to the modular double, and in Section \ref{sec:L2} a possible construction of an $L^2(GL_{q\til[q]}^+(N,\R))$ space.

\textbf{Acknowledgment.} This work was supported by World Premier International Research Center Initiative (WPI Initiative), MEXT, Japan. I would like to thank my advisor Igor Frenkel for valuable comments and suggestions to this work. I would also like to thank Alexander Goncharov for pointing out Remark \ref{clusterX} in relation to the cluster $\mathcal{X}$-varieties.

\section{Gauss decomposition for $GL_q(2)$}\label{sec:glq2}

We recall the definition of $GL_q(2)$ used in \cite{FJ,Ip}. It is a rescaled version of the usual definition of $GL_q(2)$ \cite{CP}, and has an advantage of acting naturally on the standard $L^2(\R)$ space, due to the rescaled quantum determinant \eqref{z6} which resembles the classical formula without any $q$ factor.

\begin{Def}
Let $q\in\C$ which is not a root of unity. We define $GL_q(2)$ to be the bi-algebra generated by $z_{11},z_{12},z_{21}$ and $z_{22}$ subjected to the following commutation relations:
\begin{eqnarray}
\label{z1}z_{11}z_{12}&=&z_{12}z_{11},\\
z_{21}z_{22}&=&z_{22}z_{21},\\
z_{11}z_{21}&=&q^2z_{21}z_{11},\\
z_{12}z_{22}&=&q^2z_{22}z_{12},\\
z_{12}z_{21}&=&q^2z_{21}z_{12},\\
\label{z6}det_q:=z_{11}z_{22}-z_{12}z_{21}&=&z_{22}z_{11}-z_{21}z_{12}.
\end{eqnarray}
with co-product $\D$ given by
\Eq{\D(z_{ij})=\sum_{k=1,2}z_{ik}\ox z_{kj},}
and co-unit $\e$ given by
\Eq{\e(z_{ij})=\d_{ij}.}
\end{Def}
It is possible to define the antipode $\c$ by adjoining the inverse element $\det_q\inv$, giving $GL_q(2)$ a Hopf algebra structure. However we will not use the antipode in this paper.
\begin{Rem}
It is often convenient to define the matrix of generators
$$Z:=\veca{z_{11}&z_{12}\\z_{21}&z_{22}},$$
then the co-product can be rewritten as standard matrix multiplication:
\Eq{\D\veca{z_{11}&z_{12}\\z_{21}&z_{22}}=\veca{z_{11}&z_{12}\\z_{21}&z_{22}}\ox\veca{z_{11}&z_{12}\\z_{21}&z_{22}}.}
\end{Rem}
In the papers \cite{Ip,Pu}, the Gauss decomposition of $GL_q(2)$ is studied. It can be decomposed uniquely into
\Eq{\veca{z_{11}&z_{12}\\z_{21}&z_{22}}=\veca{u_1&0\\v_1&1}\veca{1&u_2\\0&v_2},}
where the Weyl pairs $\{u_i,v_i\}_{i=1,2}$ are non-commutative variables satisfying 
\Eq{u_iv_i&=q^2v_iu_i,\\u_iv_j&=v_ju_i\mbox{\;\;\; for $i\neq j$.}} 
Denote by $\C[\T_q]$ the algebra of quantum torus:
\Eq{\C[\T_q]:=\C\<u, v\>/(uv=q^2vu)}
Then in particular, we have an embedding of the algebra $GL_q(2)$ into the algebra of quantum tori:
\Eq{GL_q(2)\to \C[\T_q]^{\ox 2}}

Later on when we specify $|q|=1$, we will introduce a star structure so that the generators $u_i,v_i$ are self-adjoint. The representations of this algebra will be discussed in detail in Section \ref{sec:posmod}.

In order to generalize this construction to the higher rank, it turns out that it is better to write it in the form:
\Eq{\label{Gauss}\veca{1&0\\v_1&1}\veca{u_1&0\\0&1}\veca{1&0\\0&v_2}\veca{1&u_2\\0&1}}
$$=\veca{u_1&0\\v_1u_1&1}\veca{1&u_2\\0&v_2},$$
which of course still satisfies the quantum relations. Finally we note that the quantum determinant $det_q$ quasi-commutes with all other variables. It is this property that motivates us to study the Gauss decomposition for $GL_q(N)$ not using the standard coordinates, but using the ``cluster variables" which we will introduce in the next section.


\section{Parametrization of $GL^+(N,\R)$}\label{sec:gln}

In classical group theory, the total positive part $GL^+(N,\R)$ is the semi-subgroup of $GL(N,\R)$ so that all the entries are positive, and all the minors, including the determinant, are also positive. There are in general two equivalent ways to realize the total positive semi-group. In \cite{Lu}, a parametrization using the Gauss decomposition is found:
\Eq{G=U_{>0}^- T_{>0} U_{>0}^+,}
where $T_{>0}$ is the diagonal matrix with positive entries $u_i$, the positive unipotent semi-subgroup $U_{>0}^+$ (and similarly for $U_{>0}^-$) is decomposed as
\Eq{\label{u0} U_{>0}^+=\prod_{k=1}^{m}e^{a_k E_{i_k}}=\prod_{k=1}^{m}(I_N+a_k E_{i_k,i_k+1}),}
where $E_{i,i+1}$ is the matrix with 1 at the position $(i,i+1)$ and 0 otherwise, and the $i_k$'s correspond to the decomposition of the longest element $w_0$ of the Weyl group $W=S_{N-1}$: 
\Eq{w_0=s_{i_1}s_{i_2}...s_{i_m}.}
Using the canonical decomposition for $w_0$:
\Eq{w_0=s_{N-1}s_{N-2}...s_2s_1s_{N-1}s_{N-2}...s_2s_{N-1}s_{N-2}...s_3...s_{N-1},}
where $s_k=(k, k+1)$ are the standard transpositions, $U_{>0}^+$ can be expressed in the form:
\footnotesize
\Eq{\label{U+}\veca{1&a_{1,1}&0&0&0\\0&1&a_{2,1}&0&0\\0&0&1&\ddots&0\\0&0&0&\ddots&a_{N-1,1}\\0&0&0&0&1}\veca{1&0&0&0&0\\0&1&a_{2,2}&0&0\\0&0&1&\ddots&0\\0&0&0&\ddots&a_{N-1,2}\\0&0&0&0&1}\cdots\veca{1&0&0&0&0\\0&1&0&0&0\\0&0&1&\ddots&0\\0&0&0&\ddots&a_{N-1,N-1}\\0&0&0&0&1}.}
\normalsize
The labeling is clear: $a_{m,n}$ is the entry at the $m$-th row and appears the $n$-th time from the left.
Similarly, $U_{>0}^-$ is given by the transpose of the form of $U_{>0}^+$, i.e.
\footnotesize
\Eq{\label{U-}\veca{1&0&0&0&0\\0&1&0&0&0\\0&0&1&0&0\\0&0&\ddots&\ddots&0\\0&0&0&b_{N-1,1}&1}\cdots\veca{1&0&0&0&0\\0&1&0&0&0\\0&b_{2,1}&1&0&0\\0&0&\ddots&\ddots&0\\0&0&0&b_{N-1,N-2}&1}\veca{1&0&0&0&0\\b_{1,1}&1&0&0&0\\0&b_{2,2}&1&0&0\\0&0&\ddots&\ddots&0\\0&0&0&b_{N-1,N-1}&1}.}
\normalsize

Under this parametrization, Berenstein et al. \cite{BFZ} studied the parametrization by cluster variables, in this case corresponds to the ``initial minors" of the matrix. These are the determinants of the square minors which start from either the top row or the leftmost column. More precisely, a matrix $g\in GL(N,\R)$ is totally positive if and only if all its initial minors are strictly positive. Furthermore, the initial minor can be expressed uniquely as a product of the parameters $a_{ij},b_{ij}$ and $u_i$, hence giving a 1-1 correspondence between the parametrizations. 

In the study of the quantum Gauss decomposition, it turns out that it is just enough to look at $T_{>0}U_{>0}^+$. Let us first consider $U_{>0}^+$. Denote by $x_{ij}$, $1\leq i<j\leq N$, the determinant of the initial minor with the lower right corner at the entry $(i,j)$. Following \cite{BFZ}, there is an explicit relation between $x_{ij}$ and $a_{ij}$:
\begin{Prop} We have
\Eq{a_{i,N-j}&=\frac{x_{j,i+1}x_{j-1,i-1}}{x_{j,i}x_{j-1,i}},\\
x_{i,i+j}&=\prod_{m=1}^i \prod_{n=1}^j a_{m+n-1,n}.}
Here we denote by $x_{i,i}=x_{i,0}=x_{0,j}=1$.
\end{Prop}

The above relations can be expressed schematically by the following diagram:

$$\begin{tikzpicture}[baseline={([yshift=3ex]a33.base)}]
\node (a51) at (0, 0) {$a_{51}$};
\node (a41) at +(30: 0.866) {$a_{41}$};
\node (a31) at +(30: 1.732) {$a_{31}$};
\node (a21) at +(30: 2.598) {$a_{21}$};
\node (a11) at +(30: 3.464) {$a_{11}$};
\node (a52) at (1.5, 0) {$a_{52}$};
\node (a42) at ($(a41)+(1.5,0)$) {$a_{42}$};
\node (a32) at ($(a31)+(1.5,0)$) {$a_{32}$};
\node (a22) at ($(a21)+(1.5,0)$) {$a_{22}$};
\node (a53) at (3, 0) {$a_{53}$};
\node (a43) at ($(a42)+(1.5,0)$) {$a_{43}$};
\node (a33) at ($(a32)+(1.5,0)$)  {$a_{33}$};
\node (a54) at (4.5, 0) {$a_{54}$};
\node (a44) at ($(a43)+(1.5,0)$) {$a_{44}$};
\node (a55) at (6, 0) {$a_{55}$};
\draw ([yshift=1ex]a11.north)--([xshift=1ex]a44.east)--([yshift=-1ex]a54.south)--([xshift=-1ex]a21.west)--cycle;
\node at ([yshift=3ex]a21.north) {$i$};
\node at ([yshift=3.2ex]a33.north) {$j$};
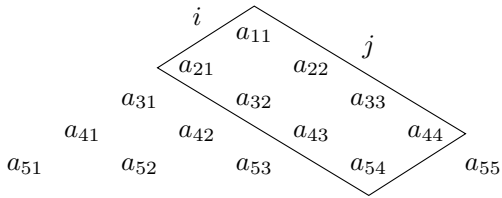
\node at (3,-0.75) {Figure 1: The cluster $x_{i,i+j}$ for $i=2,j=4$};
\end{tikzpicture}
$$

As in the $N=2$ case, we split the diagonal subgroup $T_{>0}$ into two halves:
\Eq{\label{TT}T_{>0}=T_{>0}^-T_{>0}^+:=\veca{u_1&0&0&0&0\\0&u_2&0&0&0\\0&0&\ddots&0&0\\0&0&0&u_{N-1}&0\\0&0&0&0&1}\veca{1&0&0&0&0\\0&v_1&0&0&0\\0&0&v_2&0&0\\0&0&0&\ddots&0\\0&0&0&0&v_{N-1}},}
and just consider the $v$ variables for the decomposition of the upper triangular part. Then the formulas in $T_{>0}^+U_{>0}^+$ for $a_{i,j}$ stay the same, while those for $x_{i,j}$ are modified as follows:
\Eq{x_{i,i+j}=\left(\prod_{m=1}^i \prod_{n=1}^j a_{m+n-1,n}\right)\prod_{k=1}^{i-1} v_k.}


\section{Definition of $GL_q(N)$}\label{sec:glqn}

 The quantum group $GL_q(N)$ is defined by the following relations involving $GL_q(2)$:
\begin{Def} $GL_q(N)$ is the Hopf algebra generated by $\{z_{ij}\}_{i,j=1}^N$, such that for every $1\leq i<i'\leq N,1\leq j<j'\leq N$, the minor
\Eq{\veca{z_{ij}&z_{ij'}\\z_{i'j}&z_{i'j'}}}
is a copy of $GL_q(2)$, i.e. it satisfies the corresponding relations \eqref{z1}-\eqref{z6}. Furthermore, the Hopf algebra structure is given by the same classical formula:
\Eq{\D(z_{ij})&=\sum_{k=1}^N z_{ik}\ox z_{kj},\\
\e(z_{ij})&=\d_{ij}.}
Again the antipode $\c$ can be defined by adjoining $det_q\inv$ defined below, but we will not use it in the present paper.

\end{Def}

For $GL_q(N)$, the quantum determinant is again defined using the classical formula (with no $q$ involved):
\begin{Def} We define the quantum determinant as
\Eq{\label{detn} det_q = \sum_{\s\in S_N} (-1)^\s z_{1,\s(1)}...z_{N,\s(N)},}
where $S_N$ is the permutation group. 
\end{Def}

Then it follows from \eqref{z6} and an induction argument that $det_q$ does not depend on the order of the row index, provided that all the monomials have the same order of row index.

As in the case of $GL_q(2)$, we can conveniently define the matrix of generators 
\Eq{Z:=\big(z_{ij}\big)_{i,j=1}^N}
and simply call it the $GL_q(N)$ matrix. Then from the co-associativity of the co-product $\D$, 
\Eq{\D(Z)=Z\ox Z,} we notice the following property:
\begin{Prop} If $X$ and $Y$ are $GL_q(N)$ matrices such that the generators of $X$ commute with those from $Y$, then the matrix product $G=XY$ is again a $GL_q(N)$ matrix.
\end{Prop}

Hence in order to find a Gauss decomposition $GL_q(N)=XY$ where $X$ is lower triangular and $Y$ is upper triangular, it suffices to find the corresponding matrix that satisfies the quantum relations (that any matrix can be expressed in this form is proved, for example, in \cite{DKS}). We will do this by employing the construction using the parametrizations of the totally positive matrices.


\section{Gauss-Lusztig decomposition of $GL_q(N)$} \label{sec:gauss}

Let $T^+$ and $U^+$ be given by the same matrices as in \eqref{U+} and \eqref{TT}, but instead with formal non-commuting variables $v_m, a_{mn}$ for $1\leq n\leq m\leq N-1$. We also define the cluster variables $x_{ij},1\leq i<j\leq N$ to be the quantum determinant of the initial minors of the matrix product $Z=T^+U^+$ by the determinant formula \eqref{detn}.

Then we can state our main results:
\begin{Thm}\label{main} The product $Z=T^+U^+$ is a $GL_q(N)$ matrix if and only if we have the following $q$-commutation relations between the variables given by:
\begin{itemize}
\item $a_{mn}v_m=q^2v_m a_{mn}$ for all $n$,
\item $a_{mn}a_{mn'}=q^2a_{mn'}a_{mn}$ for $n>n'$,
\item $a_{mn}a_{m-1,n'}=q^2 a_{m-1,n'}a_{mn}$ for $n\leq n'$,
\item commute otherwise.
\end{itemize}

Furthermore the variables $x_{ij}$ can be written as
\begin{eqnarray}
\label{order} x_{i,i+j}&=&\left(\prod_{m=1}^i \prod_{n=1}^j a_{m+n-1,n}\right)\prod_{k=1}^{i-1} v_k,\\
&=&(a_{11}a_{22}a_{33}...)(a_{21}a_{32}a_{43}...)...(...a_{i+j-1,j})(v_1v_2...v_{i-1})
\end{eqnarray}
in this particular order. Finally for \emph{every} $GL_q(N)$ matrix, the commutation relations between the variables $x_{ij}$ are given by
\Eq{\label{power}x_{i,i+j}x_{k,k+l}=q^{2P(i,j;k,l)} x_{l,k+l}x_{j,i+j},}
where for $j\leq l$, 
\Eq{P(i,j;k,l)=\#\{m,n| l+2\leq m+n\leq k+l+1, 1\leq m\leq i, 1\leq n\leq j\}\nonumber\\
-\#\{m,n|1\leq m+n\leq i,1\leq m\leq k,1\leq n\leq l\}}
and \Eq{P(k,l;i,j)=-P(i,j;k,l).}
\end{Thm}
\begin{Cor} Let $U^-$ and $T^-$ be defined by \eqref{U-} and \eqref{TT} so that $b_{mn}$ and $u_m$ commute with $a_{mn}$ and $v_m$. Then $\{b_{mn},u_m\inv\}$ satisfies exactly the same relations as $\{a_{mn}, v_m\}$. Let $T=T^-T^+$ be the diagonal matrices with entries $T_k=u_kv_{k-1}$ for $1\leq k\leq N$, where we denote by $v_0=u_N=1$. Then the product
\Eq{GL_q(N):=Z=U^-TU^+}
gives the Gauss-Lusztig decomposition for $GL_q(N)$ parametrized by $N^2$ variables that commute up to a factor of $q^2$. 
\end{Cor}

The $q$-commutation relations for $a_{mn}$ (and also $b_{mn}$) can be represented neatly by a diagram:
\Eq{\begin{tikzpicture}[baseline={([yshift=3ex]a33.base)}]
\node (a41) at (0, 0) {$a_{41}$};
\node (a31) at +(60: 1.5) {$a_{31}$};
\node (a21) at +(60: 3) {$a_{21}$};
\node (a11) at +(60: 4.5) {$a_{11}$};
\node (a42) at (1.5, 0) {$a_{42}$};
\node (a32) at ($(a31)+(1.5,0)$) {$a_{32}$};
\node (a22) at ($(a21)+(1.5,0)$) {$a_{22}$};
\node (a43) at (3, 0) {$a_{43}$};
\node (a33) at ($(a32)+(1.5,0)$)  {$a_{33}$};
\node (a44) at (4.5, 0) {$a_{44}$};
\node (v1) at ($(a11)+(-1.5,0)$) {$v_1$};
\node (v2) at ($(a21)+(-1.5,0)$) {$v_2$};
\node (v3) at ($(a31)+(-1.5,0)$) {$v_3$};
\node (v4) at ($(a41)+(-1.5,0)$) {$v_4$};
\node at ($(a41)+(0,-0.5)$) {$\vdots$};
\node at ($(a42)+(0,-0.5)$) {$\vdots$};
\node at ($(a43)+(0,-0.5)$) {$\vdots$};
\node at ($(a44)+(0,-0.5)$) {$\vdots$};
\draw [-angle 90] (a21) -- (a11);
\draw [-angle 90] (a43) -- (a33);
\foreach \from/\to in {a21/a11,a31/a21,a31/a22,a32/a22,a41/a31,a41/a32,a41/a33,a42/a32,a42/a33,a43/a33}{
	\draw [-angle 90] (\from) -- (\to);
}
\foreach \from/\to in {a22/a21,a32/a31,a33/a32,a42/a41,a43/a42,a44/a43}{
	\draw [-angle 90]([yshift=0.8ex]\from.west) -- ([yshift=0.8ex]\to.east);
	\draw [-angle 90]([yshift=-0.6ex]\from.west) -- ([yshift=-0.6ex]\to.east);
}
\foreach \from/\to in {a11/v1,a21/v2,a31/v3,a41/v4}{
	\draw [-angle 90](\from) -- (\to);
}
\end{tikzpicture}}
where $u\to v$ means $uv=q^2vu$, and double arrows means it $q^2$-commutes with everything in that direction. In other words, the arrows consist of all the possible left directions, and all the north-east directions going up one level. Furthermore, note that the commutation relations for $a_{mn},v_m, u_m$ and $b_{mn'}$ are just copies of the Gauss decomposition \eqref{Gauss} for $GL_q(2)$.

\begin{Rem}\label{clusterX} It was pointed out by A. Goncharov that if we make a change of variables by taking ratios of the generators:
\Eq{a'_{m,n}=\case{a_{m,1}&n=1,\\qa_{m,n}a_{m,n-1}\inv&n>1,}}
(the $q$ factor is used to preserve positivity, cf. Section \ref{sec:posmod}) then the commutation relations among the $a'_{m,n}$ variables take a more symmetric form, represented by the diagram
\Eq{\begin{tikzpicture}[baseline={([yshift=3ex]a33.base)}]
\node (a41) at (0, 0) {$a_{41}'$};
\node (a31) at +(60: 1.5) {$a_{31}'$};
\node (a21) at +(60: 3) {$a_{21}'$};
\node (a11) at +(60: 4.5) {$a_{11}'$};
\node (a42) at (1.5, 0) {$a_{42}'$};
\node (a32) at ($(a31)+(1.5,0)$) {$a_{32}'$};
\node (a22) at ($(a21)+(1.5,0)$) {$a_{22}'$};
\node (a43) at (3, 0) {$a_{43}'$};
\node (a33) at ($(a32)+(1.5,0)$)  {$a_{33}'$};
\node (a44) at (4.5, 0) {$a_{44}'$};
\node (v1) at ($(a11)+(-1.5,0)$) {$v_1$};
\node (v2) at ($(a21)+(-1.5,0)$) {$v_2$};
\node (v3) at ($(a31)+(-1.5,0)$) {$v_3$};
\node (v4) at ($(a41)+(-1.5,0)$) {$v_4$};
\node at ($(a41)+(0,-0.5)$) {$\vdots$};
\node at ($(a42)+(0,-0.5)$) {$\vdots$};
\node at ($(a43)+(0,-0.5)$) {$\vdots$};
\node at ($(a44)+(0,-0.5)$) {$\vdots$};
\node at ($(v4)+(0,-0.5)$) {$\vdots$};
\draw [-angle 90] (a21) -- (a11);
\draw [-angle 90] (a43) -- (a33);
\foreach \from/\to in {a21/a11,a11/a22,a22/a21,a31/a21,a21/a32,a32/a31,a32/a22,a22/a33,a33/a32,a41/a31,a31/a42,a42/a41,a42/a32,a32/a43,a43/a42,a43/a33,a33/a44,a44/a43}
	\draw [-angle 90] (\from) -- (\to);
\draw [-angle 90](a11)--(v1);
\foreach \from/\to in {a21/v2,a31/v3,a41/v4}{
	\draw [-angle 90]([yshift=0.8ex]\from.west) -- ([yshift=0.8ex]\to.east);
	\draw [-angle 90]([yshift=-0.6ex]\from.west) -- ([yshift=-0.6ex]\to.east);
}
\end{tikzpicture}}

This choice of generators is closely related to the Poisson structure of the cluster $\mathcal{X}$-varieties, studied for example in \cite{FG1}.
\end{Rem}
We will use several lemmas to prove the theorem.
\begin{Lem}\label{lem-order} Assume the $q$-commutation relations in Theorem \ref{main} for $a_{mn}$ and $v_m$ hold. Then \eqref{order} holds.
\end{Lem}
\begin{proof} We use the fact that, by induction, each entry $z_{ij}$ of the upper triangular matrix has a closed form expression given by
\begin{eqnarray}
z_{i,i+j}&=&v_{i-1}\sum_{1\leq t_1<t_2<...<t_{j}\leq i+j-1} (a_{i,t_1}a_{i+1,t_2}...a_{i+j-1,t_{j}})\nonumber\\
&:=&\sum_t S_{i,t}.\label{eachterm}
\end{eqnarray}
We also have $z_{i,i}=1$ and $z_{i,i-j}=0$.

Hence the quantum initial minor $x_{i,j}$ is given by sums of products of the form \Eq{\bS_{j,t}=S_{1,t_1}S_{2,t_2}...S_{j,t_j}.}

Now using the $q$-commutation relations, which say that $a_{mn}$ commutes with $a_{m'n'}$ when both $m>m'$ and $n>n'$, we can arrange the order on each monomial $\bS_{j,t}$ so that it has a ``maximal" ordering: If the product $a_{mn}a_{m'n'}$ appears in the ordering, then either $m'=m+1$ and $n'>n$, or $m'<m$. Furthermore, if the last term in $S_{k,t}$ is $a_{m,n}$, then the term $a_{m+1,n'}$ for $n'>n$ will not appear in $S_{k+1,t}$, so that nothing can commute to the front, while we can push all the $v_m$ to the back since $v_m$ commutes with $a_{m'n}$ for $m<m'$.

This ordering is unique in the sense that for every monomial where the order in which $a_{p,*}$ appears for each fixed $p$ is the same, the corresponding maximal ordering is the same. Hence the classical calculation works and all the terms will cancel, except the one with minimal lexicographical ordering. This term is precisely
$$S_{1,t_{min}}S_{2,t_{min}}...S_{i,t_{min}},$$
where
$$S_{k,t_{min}}=v_{k-1}a_{k,1}a_{k+1,2}...a_{k+j-1,j}.$$
Again each $v_{k-1}$ in each $S_k$ commutes with all the $a$'s, so we can move them towards the back, and hence giving the expression \eqref{order}.
\end{proof}

\begin{Lem}\label{lem-power} Assume the $q$-commutation relations for $a_{mn}$ and $v_m$ hold. Then \eqref{power} holds.
\end{Lem}
\begin{proof}
Using the expression given by Lem \ref{lem-order}, we can study how $x_{ij}$ and $x_{kl}$ commute. We do this by counting how many $q$-commutations it takes for a fixed $a_{m,n}$ appearing in $x_{i,i+j}$ to travel through each variable $a_{m',n'}$ in $x_{k,k+l}$.

First, notice that $a_{m,n}$ appears in $x_{k,k+l}$ only if $1\leq n\leq l$ and $0\leq m-n\leq k-1$. Now fix $m,n$ and consider $a_{mn}$. It $q^2$-commutes with $a_{m'n'}$ in $x_{k,k+l}$ when:
\begin{itemize}
\item $q^2$: $a_{m,n'}$ with $n'<n$, hence also $1\leq n'\leq l$ and $0\leq m-n'\leq k-1$. We can rewrite this as
$$A_1=\#\{n'|\max(m+1-k,1)\leq n'\leq \min(l,n-1,m)\},$$
\item $q^{-2}$: $a_{m,n'}$ with $n'>n$, hence also $1\leq n'\leq l$ and $0\leq m-n'\leq k-1$ which reduces to
$$A_2=\#\{n'|\max(m+1-k,n+1)\leq n'\leq \min(l,m)\},$$
\item $q^2$: $a_{m-1,n'}$ with $n'\geq n$, hence
$$A_3=\#\{n'|\max(m-k,n)\leq n'\leq \min(l,m-1)\},$$
\item $q^{-2}$: $a_{m+1,n'}$ with $n'\leq n$, hence
$$A_4=\#\{\n'|\max(m-l+2,1)\leq n'\leq \min(l,n,m+1)\}.$$
\end{itemize}
Hence the amount of $q^2$ powers picked up is just the signed sum of the count above. By a case by case study, these expressions can be simplified:
$$A_3-A_2=\left\{\begin{array}{cc}1& m+n\geq l+2,n+m\leq k+l+1, n\leq l\\0& \mbox{otherwise,}\end{array}\right.$$
$$A_1-A_4=\left\{\begin{array}{cc}1&n\geq l+1, k+1\leq m+n\leq k+l\\ -1& m+n\leq k,1\leq n\leq l\\0&\mbox{otherwise.}\end{array}\right.$$
Hence, the total amount of power picked up after summing all $m,n$ is given by
\begin{eqnarray*}
&&\#\{l+2\leq m+n\leq k+l+1,n\leq l\}+\#\{k+1\leq m+n\leq l+k, l+1\leq n \}\\
&&\tab-\#\{m+n\leq k,1\leq n\leq l\},
\end{eqnarray*}
subject to $1\leq m\leq i, 1\leq n\leq j$.

Let us assume $j\leq l$. Then $n\leq j\leq l$, and the expression can be simplified to
$$\#\{l+2\leq m+n\leq k+l+1\}-\#\{m+n\leq k\}.$$
subject to $1\leq m\leq i, 1\leq n\leq j$.
This takes care of $a_{mn}$. 

We still need to calculate those for $v_m$. Since there is only one $v_m$ appearing in $x_{k,k+l}$ for each $1\leq m\leq k$, we just need to count how many $a_{mn}$'s with index $m\leq k+1$ are there. Hence using the renamed $a_{m+n-1,n}$ the condition is 
$$\#\{m,n|m+n\leq k, 1\leq m\leq i, 1\leq n\leq j\},$$
and this is the amount of $q^{2}$ picked up, hence canceled with the last term in the previous calculation.

Similarly considering the other direction, the amount of $q^{-2}$ picked up is 
$$\#\{m,n|m+n\leq i, 1\leq m\leq k, 1\leq n\leq l\}.$$
Hence we arrive at our formula.
\end{proof}

\begin{Lem}\label{lem-power2} We have 
\Eq{P(i,j;k,l)=P(i,j-1,k,l-1).}
\end{Lem}
\begin{proof} This is done by simple counting. Assume $j\leq l$. Let us compare the difference between the corresponding terms of the $P$ function. We have for the second term:
\begin{eqnarray*}
&&\#\{m,n| 1\leq m+n\leq i,1\leq m\leq k,1\leq n\leq l\} \\
&&\tab- \#\{m,n| m+n\leq i,1\leq m\leq k,1\leq n\leq l-1\}\\
&=&\#\{m,n| 1\leq m+l\leq i,1\leq m\leq k\},\\
\end{eqnarray*}
while for the first term we have
\begin{eqnarray*}
&&\#\{m,n| l+2\leq m+n\leq k+l+1, 1\leq m\leq i, 1\leq n\leq j\}\\
&&\tab-\#\{m,n| l+1\leq m+n\leq k+l, 1\leq m\leq i, 1\leq n\leq j-1\}\\
&=&\#\{m,n| l+2\leq m+n\leq k+l+1, 1\leq m\leq i, 1\leq n\leq j\}\\
&&\tab-\#\{m,n| l+2\leq m+n\leq k+l+1, 1\leq m\leq i, 2\leq n\leq j\}\\
&=&\#\{m| l+2\leq m+1\leq k+l+1, 1\leq m\leq i\}\\
&=&\#\{m| l+2\leq m+l+1\leq k+l+1, 1\leq m+l\leq i\}\\
&=&\#\{m| 1\leq m\leq k, 1\leq m+l\leq i\}.\\
\end{eqnarray*}
Hence the amounts cancel.
\end{proof}

\begin{proof}[Proof of Theorem \ref{main}] We will prove the theorem by induction. When $N=2$ it is just
$$\veca{1&0\\0&v_1}\veca{1&a_{11}\\0&1}$$
with $a_{11}v_1=q^2 v_1 a_{11}$. Hence this case holds trivially.

Assume everything hold for $\dim=N-1$. 

For $\dim=N$, first we notice that $a_{N-1,N-1}$ commutes with $a_{ii}$ for $i<N-1$ by looking at the entry $z_{1,i+1}=a_{11}a_{22}...a_{ii}$, which commutes with each other by the $GL_q(2)$ relations.

Next we notice that the cluster variables for a general $GL_q(N)$ matrix depend only on the variables appearing in $T^+U^+$, since we assumed that the lower triangular matrix $U^-T^-$ commutes with $T^+U^+$. Hence the relations between $x_{i,i+j}$ which hold for $T^+U^+$ will also hold for $GL_q(N)$.

Now for a general cluster variable $x_{k,N}$ in the new rank, we know from Lemma \ref{lem-order} that $a_{N-1,N-k}$ is the only new term appearing. Hence the commutation relations between $a_{N-1,N-k}$ and $a_{i+j-1,j}$ is equivalent to the commutation relations between $x_{k,N}$ and $x_{i,i+j}$ by induction on new terms. Now consider the $(N-1)\times (N-1)$ minor corresponding to $x_{N-1,N}$. This by definition satisfies the $GL_q(N-1)$ relations, and in particular the commutation relation between
$x_{k,N}$ and $x_{i,i+j}$ should be the same as the relation between $x_{k,N-1}$ and $x_{i,i+j-1}$. However, this is precisely the statement proved in Lemma \ref{lem-power2}.
\end{proof}

The above relations can be generalized to arbitrary reduced expression for $w_0$ as follows.
Let $(a,b,c)$ and $(a',b',c')$ be positive $q$-commuting variables such that 
\Eq{\label{diagram}\begin{tikzpicture}[baseline={([yshift=3ex]c.base)}]
\node (a) at (0, 0) {$a$};
\node (b) at +(50: 1.17) {$b$};
\node (c) at +(0: 1.5) {$c$};
\foreach \from/\to in {c/a, a/b}
\draw [-angle 90] (\from) -- (\to);
\end{tikzpicture}
\tab\mbox{and}\tab
\begin{tikzpicture}[baseline={([yshift=3ex]b.base)}]
\node (a) at (0, 0) {$a'$};
\node (b) at +(-50: 1.17) {$b'$};
\node (c) at +(0: 1.5) {$c'$};
\foreach \from/\to in {b/c, c/a}
\draw [-angle 90] (\from) -- (\to);
\end{tikzpicture}}
where again $u\to v$ means $uv=q^2vu$.

Then as in \eqref{u0}, the products $$x_2(a)x_1(b)x_2(c)=x_1(a')x_2(b')x_1(c')$$ form a copy of $U^+$ of the Gauss Decomposition of $GL_q(3)$ corresponding to the reduced expressions
$$w_0=s_2s_1s_2=s_1s_2s_1,$$ where 
\begin{eqnarray*}
a'&=&(a+c)\inv cb=bc(a+c)\inv,\\
b'&=&a+c,\\
c'&=&(a+c)\inv ab=ba(a+c)\inv,
\end{eqnarray*}
and this map 
\Eq{\phi:(a,b,c)\mapsto(a',b',c')\label{inphi}}
 is an involution between $(a,b,c)\corr(a',b',c')$. In particular, we see that by applying this transformation to any three consecutive variables $a_{mn}$ corresponding to the sub-word of the form $s_is_js_i$ with $i$ adjacent to $j$, all the arrows in the diagram \eqref{diagram} are preserved. Applying this transformation, we can deduce all quantum Lusztig's variables for arbitrary reduced expression for $w_0$. Hence we can restate the commutation relations in Theorem \ref{main} as follows:

\begin{Thm}\label{GLcomtrans} Let $a_{i_n,m}$ be the coordinates of $U^+$ corresponding to the reduced expression of $w_0=s_{i_1}...s_{i_n}$. Then the product $T^+U^+$ is a $GL_q(N)$ matrix if and only if for any $|i-j|=1$, the coordinates $\{v_i,v_j,a_{i,m},a_{j,n},a_{i,k}\}$ form a copy of $GL_q(3)$, where $\{a_{i,m},a_{j,n},a_{i,k}\}$ appear in this exact order in the parametrization of $U^+$. In other words, we have
\Eq{\begin{tikzpicture}[baseline={([yshift=3ex]c.base)}]
\node (a) at (0, 0) {$a_{i,m}$};
\node (b) at +(50: 1.5) {$a_{j,n}$};
\node (c) at +(0: 1.93) {$a_{i,k}$};
\node (vi) at (-2,1.15) {$v_i$};
\node (vj) at (-2,0) {$v_j$};
\node (dot) at (-0.5,1.15) {$\cdots$};
\node (dot2) at (-0.5,0) {$\cdots$};
\foreach \from/\to in {c/a, a/b}
	\draw [-angle 90] (\from) -- (\to);
\draw [-angle 90]([yshift=0.8ex]dot2.west) -- ([yshift=0.8ex]vj.east);
\draw [-angle 90]([yshift=-0.6ex]dot2.west) -- ([yshift=-0.6ex]vj.east);
\draw [-angle 90](dot) --(vi);
\end{tikzpicture}}
\end{Thm}

\section{Embedding into the algebra of quantum tori}\label{sec:qtori}
We would like to find an embedding of $GL_q(N)$ into copies of the algebra of quantum tori $\C[\T_q]$. Hence the remaining task is to find an appropriate realization of the generators $a_{mn},v_m$ using several copies of the Weyl pair $\{u,v\}$ satisfying $uv=q^2vu$.

\begin{Thm}\label{qtori1} There is an embedding of algebra $$T^+U^+\to \C[\T_q]^{\ox \frac{N^2+N-4}{2}}$$
given by
\Eq{v_m&\mapsto v_m\\
a_{mn}&\mapsto u_m \left(\prod_{k=n}^{m-1}v_{m-1,k}\right)\left(\prod_{l=1}^{n-1}v_{m,l}\right)u_{m,n}.}
where $T^+U^+$ is now realized as the algebra generated by $\{a_{mn},v_m\}$ satisfying the relations from Theorem \ref{main}, and $\C[\T_q]^{\ox \frac{N^2+N-4}{2}}$ is generated by the Weyl pairs $\{u_m,v_m\}$ and $\{u_{mn},v_{mn}\}$ for $1\leq n\leq m\leq N-1$, where we have omitted the last set of generators $\{u_{N-1,N-1},v_{N-1,N-1}\}$. (We define $u_{N-1,N-1}:=1$ in the formula).

Similarly, for $U^-T^-$ generated by $\{b_{mn},u_m\}$, we have the embedding
$$U^-T^-\to \C[\T_q]^{\ox \frac{N^2+N-4}{2}}$$
given by
\Eq{u_m&\mapsto u'_m\\
b_{mn}&\mapsto v'_m \left(\prod_{k=n}^{m-1}v'_{m-1,k}\right)\left(\prod_{l=1}^{n-1}v'_{m,l}\right)u'_{m,n}.}
where the generators $\{u',v'\}$ (with same indexing above) commute with $\{u,v\}$ used above.

Together, this gives an embedding of algebra
\Eq{GL_q(N)=U^-T^-T^+U^+\longrightarrow \C[\T_q]^{\ox N^2+N-4}.}
\end{Thm}
\begin{proof} The proof is straightforward to check, since for the $u$ variables only $u_m$ and $u_{mn}$ appear in $a_{mn}$. Hence we just need to count, at most once, how many $v_{mn}$ appears in another variables. 
\end{proof}
\begin{Rem} This embedding resembles the Drinfeld double construction, which reads
\Eq{\cD(\cU_q(\fb))=\cU_q(\g)\ox \cU_q(\fh).}
With our assignment for $U^-T^-T^+U^+$, we can actually combine the diagonal variables (hence ``modding" out $\fh$) as follows:
\Eq{V_i=u'_iv_{i-1}\tab 1\leq i\leq N,}
\Eq{U_{i}=u_{i-1}={v'_i}\inv,}
(where $v_0=u'_N=1$) which gives an embedding of $GL_q(N)$ using only $(N^2-2)$ copies of $q$-tori (with inverses adjoined).
\end{Rem}
This is just one example of realizing the quantum variables where we can actually write down explicit expressions. In fact the minimal amount needed can be substantially smaller:
\begin{Thm}\label{qtori2}
The minimal amount of $q$-tori needed to realize $T^+U^+$ is given by
$\lfloor\frac{N^2}{4}\rfloor$
and the full group $GL_q(N)$ can be embedded into $\C[\T_q]^{\ox \lfloor\frac{N^2}{2}\rfloor}$.
\end{Thm}
\begin{proof}
 Consider the symplectic form on the variables $\{a_{mn},u_m\}$ defined by
\Eq{\<x,y\>=\case{1&xy=q^2yx\\0&xy=yx\\-1 &xy=q^{-2}yx.}}
Then the minimal amount of $q$-tori needed to realize such relations can be found by finding the signature of this form. The skew-symmetric matrix of size $\frac{N^2+N-2}{2}$ encoding this form is actually quite simple. If we index the variables by $$a_{11},v_{1},a_{21},a_{22},v_{2},a_{31},...,$$
the upper triangular part of the matrix is given by $m+1$ consecutive 1's to the right starting at the first off diagonal entry corresponding to $a_{m*}$, truncated at the boundary, and zero otherwise. The full matrix is then obtained by anti-symmetrizing it, and we can find its kernel by elementary operations.
\end{proof}
In principle, it is possible to find the decomposition into the $q$-tori by diagonalizing the skew-symmetric matrix, corresponding to the symplectic form, into blocks $\veca{0&1\\-1&0}$ and read out the transformation. 

\begin{Ex}As an example, we illustrate the cases up to $N=6$, giving the embedding of $T^+U^+$ into
$\lfloor\frac{N^2}{4}\rfloor = 9$ copies of the algebra of $q$-tori generated by $\{U_i,V_i\}$ without any powers or inverses (for $N<6$ we ignore the extra tori):

\begin{eqnarray*}
a_{11}&=&U_1\\
u_1&=&V_1\\\\
a_{21}&=&V_1U_2\\
a_{22}&=&qV_2U_2U_3\\
u_2&=&V_2\\\\
a_{31}&=&V_2U_3U_4\\
a_{32}&=&V_3U_4U_5\\
a_{33}&=&qV_4U_4U_5\\
u_3&=&V_4\\\\
a_{41}&=&V_4U_5U_6V_7V_8\\
a_{42}&=&qV_5U_5U_6V_7\\
a_{43}&=&U_3V_5U_6\\
a_{44}&=&qV_6U_6\\
u_4&=&V_6\\\\
a_{51}&=&V_6U_9\\
a_{52}&=&qV_6U_8V_9U_9\\
a_{53}&=&V_6U_7V_7V_9U_9\\
a_{54}&=&U_5V_6V_7U_8V_8V_9U_9\\
a_{55}&=&qV_8V_9U_9\\
u_5&=&V_9.\\
\end{eqnarray*}
Again the extra $q$ factors are introduced for positivity, as explained in the next section.
\end{Ex}
\begin{Con} It is possible to decompose each $a_{mn}$ into a product of \emph{single} $U_i$ and $V_i$ with the minimal amount of copies. This means that we have an embedding of $GL_q(N)$ into the \emph{polynomial} algebra generated by the minimal amount of $U_i$ and $V_i$, where the matrix entries $z_{ij}$ are expressed only in terms of polynomials of $U_i$ and $V_i$ with coefficients of the form $+q^n$.
\end{Con}


\section{Positivity and the modular double}\label{sec:posmod}
Let $q=e^{\pi ib^2}$ with $0<b^2<1, b\in\R\setminus\Q$ so that $|q|=1$ is not a root of unity. In this section we introduce the notion of the positive quantum semi-group $GL_q^+(N,\R)$ and Faddeev's notion of the modular double of $GL_q^+(N,\R)$. The basic idea is to represent the generators $z_{ij}$ in terms of positive essentially self adjoint operators acting on certain Hilbert space, where these operators are necessarily unbounded. 

First we introduce the definition of an integrable representation of the canonical commutation relation defined by \cite{Sch}:
\begin{Def}\label{Sch}Let $X,Y$ be positive essentially self-adjoint operators acting on a Hilbert space $\cH$. The $q$-commutation relation $``uv=q^2vu"$ is defined to be
\Eq{u^{is}v^{it}=q^{-2st}v^{it}u^{is}}
for any $s,t\in\R$ as relations of bounded operators, where $u^{is}$ and $v^{it}$ are unitary operators on $\cH$ by the use of functional calculus.
\end{Def}
A canonical irreducible integrable representation of $uv=q^2vu$ is given by
\Eq{\label{uv}u=e^{2\pi bx}, \tab v=e^{2\pi bp},}
where $p=\frac{1}{2\pi}\del[,x]$ acting as unbounded operators on $L^2(\R)$.

On the other hand, the idea of the modular double of the Weyl pair $\{u,v\}$ is introduced by Faddeev in \cite{Fa1}. There it is suggested that for positive self-adjoint operators $u,v$ as in \eqref{uv}, one has to consider also the operators given by
\Eq{\til[u]:=u^{\frac{1}{b^2}},\tab \til[v]:=v^{\frac{1}{b^2}}}
so that \Eq{\til[u]\til[v]=\til[q]^2\til[v]\til[u],\tab \til[q]=e^{\pi ib^{-2}},}
and $\{u,v\}$ commute with $\{\til[u],\til[v]\}$ in the weak sense (the spectrum do not commute), so that the operators generated by $\{u,v,\til[u],\til[v]\}$ acting on $L^2(\R)$ is algebraically irreducible. This idea is subsequently extended to the modular double of quantum groups \cite{Fa2}.

The key technical tool in the above contexts is given by the following Lemma introduced by Volkov \cite{Vo}, and the self-adjointness is analyzed in \cite{Ru}, see also \cite{BT,Ip}:
\begin{Lem}\label{b2lem} If $u$ and $v$ are positive essentially self-adjoint operators such that $uv=q^2vu$ for $q=e^{\pi ib^2}$ in the sense above, then $u+v$ is also positive essentially self-adjoint, and we have
\Eq{(u+v)^{\frac{1}{b^2}}=u^{\frac{1}{b^2}}+v^{\frac{1}{b^2}}.}
Hence by induction, if we have $u_iu_j=q^2u_ju_i$ for every $i<j$, then the sum
$$z=\sum_i u_i$$
is positive essentially self-adjoint, and we have
\Eq{z^{\frac{1}{b^2}}=\sum_i u_i^{\frac{1}{b^2}}.}
\end{Lem}
Hence using the Gauss-Lusztig decomposition of $GL_q(N)$ defined in Theorem \ref{main}, we can define
\begin{Def} $GL_q^+(N,\R)$ is the algebra generated by positive self-adjoint operators $\{a_{mn},u_m,v_m,b_{mn}\}$ so that the $q$-commutation relations are satisfied in the sense of Definition \ref{Sch} above. These operators are acting by the Weyl pair $u=e^{2\pi bx}, v=e^{2\pi bp}$ on the Hilbert space  $\cH=L^2(\R^{\lfloor\frac{N^2}{2}\rfloor})$ using Theorem \ref{qtori2}.
\end{Def}

Using Lemma \ref{b2lem}, the notion of $GL_q^+(N,\R)$ as the $q$-analogue of the classical totally positive semi-group is justified:
\begin{Cor}\label{cor1} Under the Gauss-Lusztig decomposition, the generators $z_{ij}$, as well as the initial minors $x_{ij}$ and the quantum determinant $det_q$ are represented by positive essentially self-adjoint operators.
\end{Cor}
\begin{proof} The only issue concerns the essential self-adjointness of $z_{ij}$. From the expression \eqref{eachterm}, we note that $$S_{i,t}S_{i,t'}=q^2S_{i,t'}S_{i,t},$$
whenever $S_{i,t'}$ appears later than $S_{i,t}$ in the sum. Hence using Lemma \ref{b2lem}, we conclude that each $z_{i,i+j}$ is positive and essentially self-adjoint.
\end{proof}

Let $\til[q]=e^{\pi ib^{-2}}$ and define $GL_{\til[q]}^+(N,\R)$ to be the algebra generated by the positive self-adjoint operators $\{\til[a_{mn}],\til[u_m],\til[v_m],\til[b_{mn}]\}$ where $\til[X]:=X^{\frac{1}{b^2}}$. Then the last statement of Lemma \ref{b2lem} establishes the transcendental relations between the two parts of the modular double $GL_{q\til[q]}^+(N,\R):=GL_q^+(N,\R)\ox GL_{\til[q]}^+(N,\R)$:

\begin{Thm}  The generators $\til[z]_{ij}$ represented by the Gauss-Lusztig decomposition is related to $z_{ij}$ by
\Eq{\label{zz}\til[z_{ij}]=z_{ij}^{\frac{1}{b^2}},}
which is well-defined as positive essentially self-adjoint operators. Furthermore the co-product is preserved:
\Eq{(\D z_{ij})^{\frac{1}{b^2}}=\D\til[z_{ij}],}
and $z_{ij}$ commutes (weakly) with $\til[z_{ij}]$.
\end{Thm}
\begin{proof} Equation \eqref{zz} follows immediately from the proof of Corollary \ref{cor1}. The co-product is preserved because 
\Eq{\D(z_{ij})=\sum_k z_{ik}\ox z_{kj},}
and $$(z_{ik}\ox z_{kj})(z_{ik'}\ox z_{k'j})=q^2(z_{ik'}\ox z_{k'j})(z_{ik}\ox z_{kj})$$
whenever $k<k'$, hence we can apply Lemma \ref{b2lem} and induction to obtain 
$$\D(z_{ij})^{\frac{1}{b^2}}=\sum_k z_{ik}^{\frac{1}{b^2}}\ox z_{kj}^{\frac{1}{b^2}}=\D(\til[z_{ij}]).$$
Since $\{u,v\}$ commute weakly with $\{\til[u],\til[v]\}$, the last statement follows by the Gauss-Lusztig decomposition.
\end{proof}


\section{$GL_{q\til[q]}^+(N,\R)$ in the $L^2$ setting}\label{sec:L2}
In the final section we generalize the approach from \cite{Ip} to higher rank, where the $L^2$-space for $GL_{q\til[q]}^+(2,\R)$ is introduced. The $L^2$-norm comes from the classical counterpart of the Haar measure. In the classical case, under the decomposition given by \eqref{Gauss}, the Haar measure induced on the coordinates is given by
\Eq{dg=\frac{du_1}{u_1}dv_1\frac{dv_2}{v_2}du_2.}
When the original positive variable has $dx$ as the measure, the Mellin transform gives
\Eq{\|f(x)\|_{dx}^2=\left\|\int_{\R+i0} f(s)x^{is}ds\right\|^2=\int_\R |f(s+\frac{i}{2})|^2 ds.\label{L2norm}}
Similarly, when the measure of the original positive variable is given by $\frac{dy}{y}$, the Mellin transformed measure remains unchanged:
\Eq{\|f(y)\|_{\frac{dy}{y}}^2=\left\|\int_\R f(s)y^{is}ds\right\|^2 = \int_\R |f(s)|^2ds.\label{L2norm2}}
It is shown in \cite{Ip} that in the quantum setting, using certain GNS representation of $GL_{q\til[q]}^+(2,\R)$ in the $C^*$-algebraic setting, there exists a Haar weight which induces an $L^2$-norm on entire functions with 4 \emph{classical} coordinates $f(u_1,v_1,u_2,v_2)$, where the right hand side of \eqref{L2norm} is replaced by
\Eq{\int_\R \left|f(s+\frac{iQ}{2})\right|^2 ds}
with $Q=b+b\inv$, while the right hand side of \eqref{L2norm2} remains unchanged.

For higher rank, we notice that the Haar measure for the usual totally positive parametrization 
\Eq{U_{>0}^-(b_{ij})T_{>0}(u_{i})U_{>0}^+(a_{ij})}
is given by
\Eq{\prod_{k=1}^N\frac{du_k}{u_k}\prod_{1\leq j\leq i\leq N-1}(a_{ij}b_{ij})^{N-1-i}da_{ij}db_{ij}.}
This is obtained by calculating the Jacobian of the change of variables \eqref{eachterm} with coordinates of the usual Gauss decomposition, and using the fact that the Haar measure induces a multiplicative measure on the diagonal matrices, and the standard $L^2$ measure on the upper/lower triangular unipotent matrices.

Now let us introduce a change of variables for our cluster variables:
\Eq{X_{ij}:=\left\{\begin{array}{cc}x_{ii}x_{i-1,i-1}\inv&i=j\\x_{ij}x_{ii}\inv&i< j\\x_{ij}x_{jj}\inv&i> j,\end{array}\right.}
i.e. we use only $u_i$ from the diagonal matrices, and the cluster variables of the upper/lower triangular unipotent matrices in terms of products of $a$'s or $b$'s only. Then we have

\begin{Prop} The (classical) Haar measure on $GL^+(N,\R)$ is given by:
\Eq{\left(\prod_{i,j=1}^{N-1}\frac{dX_{ij}}{X_{ij}}\right)\left(\prod_{k=1}^{N-1}dX_{N,k}dX_{k,N}\right) \frac{dX_{NN}}{X_{NN}}.}
\end{Prop}

Hence following the idea in the case of $GL_{q\til[q]}^+(2,\R)$ we can define $C_\oo(GL_{q\til[q]}^+(N,\R))$ and $L^2(GL_{q\til[q]}^+(N,\R))$ as follow:

\begin{Def} The $C^*$-algebra $C_\oo(GL_{q\til[q]}^+(N,\R))$ is the norm closure of operators of the form:
\Eq{F:=\int_{\R^{N^2}} f(s_{11},...,s_{NN})\prod_{m,n=1}^{N}(X_{mn})^{ib\inv s_{mn}} \prod_{m,n=1}^N ds_{mn},}
where $f(s_{11},...,s_{NN})$ are smooth analytic rapidly decreasing functions in each variable $s_{ij}\in~\R$, each $X_{mn}^{ib\inv s_{mn}}$ is realized as a unitary operator defined by the formula in Theorem \ref{qtori2}, and the $C^*$-norm is defined as the operator norm. Furthermore, the $C^*$-algebra corresponding to difference choices of realization of $X_{mn}$ are isomorphic.
\end{Def}
\begin{Rem} It will be interesting to put $C_\oo(GL_{q\til[q]}^+(N,\R))$ in the context of locally compact quantum group, for example in the sense of \cite{KV}. Although it is known in the case when $N=2$ \cite{Ip}, in general it is not clear how the co-product $\D$ defined on $z_{ij}$ translates to the generators $X_{mn}$, hence the main difficulty will be to show the density conditions involving $\D$.
\end{Rem}

For the sake of harmonic analysis, we can also define an $L^2$-space where it does not depend on the choice of embedding of $GL_q^+(N,\R)$ at all.
\begin{Def}
We define $L^2(GL_{q\til[q]}^+(N,\R))\simeq L^2(\R^{N^2})$ by giving an $L^2$-norm for the (rapidly decreasing entire) functions $F$:
\Eq{\|F\|^2 := \int_{\R^{N^2}}\left|f(s_{ij}+(\d_{i,N}-\d_{j,N})^2\frac{iQ}{2})\right|^2 \prod_{i,j=1}^N ds_{ij}.}
and taking the $L^2$-completion.
\end{Def}
A natural class of representations for split real quantum groups $U_{q\til[q]}(\g_\R)$ for arbitrary type simple Lie algebra $\g$, called the positive principal series representations, generalizing the $\sl(2,\R)$ case is introduced in \cite{FI} and constructed in \cite{Ip2,Ip3}. It is conjectured in \cite{PT1} and shown in \cite{Ip} that the left and right regular representations of $U_{q\til[q]}(\gl(2,\R))$ acting on $L^2(GL_{q\til[q]}^+(2,\R))$ naturally decompose as a direct integral of tensor product of the positive representations $\cP_{\l,s}$:
\Eq{L^2(GL_{q\til[q]}^+(2,\R))\simeq  \int_\R^\o+ \int_{\R_+}^\o+\cP_{\l,s}\ox\cP_{\l,-s}d\mu(\l)ds}
where $\mu(\l)$ is expressed in terms of the quantum dilogarithm function. Hence analogous to the harmonic analysis of $L^2(GL_{q\til[q]}^+(2,\R))$, it is natural to ask the following question:
\begin{Con} Do the left and right regular representations of $U_{q\til[q]}(\gl(N,\R))$ on the $L^2(GL_{q\til[q]}^+(N,\R))$ space decompose as a direct integral of tensor product of the positive principal series representations of $U_{q\til[q]}(\gl(N,\R))$?

Furthermore, can analogous statements be defined for $U_{q\til[q]}(\g_\R)$ for arbitrary type simple Lie algebra $\g$?
\end{Con}



\begin{thebibliography}{99}
\bibitem{BFZ}
A. Berenstein, S. Fomin, A. Zelevinsky,
\textit{Parametrizationf of canonical bases and totally positive matrices},
Adv. in Math., \textbf{122}, (1996), 49-149
\bibitem{BZ}
A. Berenstein, A. Zelevinsky,
\textit{Quantum cluster algebras},
Adv. in Math., \textbf{195}, (2005), 405-455
\bibitem{BT}
  A.G. Bytsko, K. Teschner,
  \textit{R-Operator, co-product and Haar-measure for the modular double of $U_q(\sl(2,\R))$},
  Comm. Math. Phys. \textbf{240}, (2003), 171-196
\bibitem{CP}
V. Chari, A. Pressley,
\textit{A guide to quantum group},
Cambridge University Press, (1995)
\bibitem{DKS}
E. V. Damaskinsky, P. P. Kulish, M. A. Sokolov,
\textit{Gauss decomposition for quantum groups and supergroups},
arXiv:q-alg/9505001v1, (1995)
\bibitem{Fa1}
  L.D. Faddeev,
  \textit{Discrete Heisenberg-Weyl group and modular group},
  Lett. Math. Phys. \textbf{34}, (1995), 249-254
\bibitem{Fa2}
  L.D. Faddeev,
  \textit{Modular double of quantum group},
  arXiv:math/9912078v1 [math.QA], (1999)
\bibitem{FG1}
  V. V. Fock, A.B. Goncharov,
  \textit{Cluster $\mathcal{X}$-varieties, amalgamation, and Poisson-Lie groups},
 Prog. in Math., \textbf{253}, (2006), 27-68
\bibitem{FG2}
  V. V. Fock, A.B. Goncharov,
  \textit{The quantum dilogarithm and representations of quantized cluster varieties},
  Inv. Math., \textbf{175} (2), (2007)  , 223-286
  \bibitem{FI}
  I. B. Frenkel, I. Ip,
  \textit{Positive representation of split real quantum groups and future perspectives},
  arXiv:1111:1033, (2011)
  \bibitem{FJ}
  I. B. Frenkel, M. Jardim,
  \textit{Quantum instantons with classical moduli spaces},
  Comm. Math. Phys., \textbf{237}(3), (2003), 471-505
 \bibitem{Ip}
  I. Ip,
  \textit{Representation of the quantum plane, its quantum double and harmonic analysis on $GL_q^+(2,\R)$},  arXiv:1108.5365, (2011) 
\bibitem{Ip2}
I. Ip,
\textit{Positive Representations of Split Real Simply-laced Quantum Groups}, arXiv:1203:2018, (2012)
\bibitem{Ip3}
I. Ip,
\textit{Positive representations of split real quantum groups of type $B_n$, $C_n$, $F_4$, and $G_2$}, arXiv:1205.2940, (2012)
 \bibitem{KV}
  J. Kustermans and S. Vaes,
  \textit{Locally compact quantum groups},
  Ann. Sci. Ecole Norm. Sup. (4) \textbf{33}, (2000), 837-934
\bibitem{Lu}
G. Lusztig,
\textit{Total positivity in reductive groups},
 in: `` Lie theory and geometry: in honor of B. Kostant,"  Progr. in Math.\textbf{123}, Birkhauser, (1994), 531-568
\bibitem{Sch}
K. Schm\"{u}dgen,
\textit{Operator representations of $\R_q^2$},
Publ. RIMS Kyoto Univ. \textbf{29}, (1993), 1030-1061
\bibitem{PW}
  P. Podle\'{s}, S.L. Woronowicz,
  \textit{Quantum deformation of Lorentz group},
  Commun Math. Phys. 130, (1990), 381-431 

\bibitem{PT1}
  B. Ponsot, J. Teschner,
  \textit{Liouville bootstrap via harmonic analysis on a noncompact quantum group},
  arXiv: hep-th/9911110, (1999)

\bibitem{Pu}
  W. Pusz,
  \textit{Quantum $GL(2,\C)$ group as double group over $'az+b'$ quantum group},
  Reports on Math. Phys., \textbf{49},  (2002), 113-122
\bibitem{Ru}
  S.N.M. Ruijsenaars,
  \textit{A unitary joint eigenfunction transform for the $A\D O$'s $\exp(ia_\pm d/dz)+\exp(2\pi z/a_\mp)$},
  J. Nonlinear Math. Phys. \textbf{12} Suppl. 2, (2005), 253-294
\bibitem{Vo}
  A. Yu. Volkov,
  \textit{Noncommutative hypergeometry},
  Comm. Math. Phys. \textbf{258}(2), (2005), 257-273
\bibitem{WZ}
S.L. Woronowicz, S. Zakrzewski,
\textit{Quantum Lorentz group having Gauss decomposition property},
Publ. RIMS, Kyoto Univ. \textbf{28} (1992), 809-824
\end{thebibliography}
\end{document}